\documentclass[article]{amsart}
\usepackage{url,xcolor}
\usepackage{amscd}
\usepackage[centertags]{amsmath}
\usepackage{latexsym}
\usepackage{amsfonts}
\usepackage{amssymb}
\usepackage{amsthm}
\usepackage{newlfont}
\usepackage{graphics}
\usepackage{graphicx}
\usepackage[all]{xy}
\usepackage[dvipdfmx]{attachfile2}
\def\N{\mathbb{N}}

\theoremstyle{plain}
\newtheorem{thm}{Theorem}

\newtheorem{cor}[thm]{Corollary}
\newtheorem{lem}[thm]{Lemma}
\newtheorem{prop}[thm]{Proposition}

\newtheorem{con}[thm]{Conjecture}

\theoremstyle{remark}
\newtheorem{rem}{Remark}

\theoremstyle{definition}

\begin{document}

\date{June 2014}

\title{Non-periodic continued fractions for quadratic irrationalities}

\author{Michael O. Oyengo}
\address{Department of Mathematics \\
University of Illinois\\
1409 W. Green Street \\
Urbana, IL 61801}
\email{mchlyng2@illinois.edu}

%\author{Kenneth B. Stolarsky}
%\address{Department of Mathematics \\
%University of Illinois\\
%1409 W. Green Street \\
%Urbana, IL 61801}
%\email{stolarsk@math.uiuc.edu}
%\thanks{Partially supported by grant DGICYT PB98-0618}

%\subjclass{57R30, 55N30}

\begin{abstract}
A well known theorem of Lagrange states that the simple continued fraction of a real number $\alpha$ is periodic if and only if $\alpha$ is a quadratic irrational. We examine non-periodic and non-simple continued fractions formed by two interlacing geometric series and show that in certain cases they converge to quadratic irrationalities. This phenomenon is connected with certain sequences of polynomials whose properties we examine further.
\end{abstract}

\maketitle


%%%%%%%%%%%%%%%%%%%%%%%%%%%%%%%%%%%%%%%%%%%%%%%%%%%%%%%%%%%%%%%%%%%%%%%%%
% Macros
%%%%%%%%%%%%%%%%%%%%%%%%%%%%%%%%%%%%%%%%%%%%%%%%%%%%%%%%%%%%%%%%%%%%%%%%%

\newcommand\sfrac[2]{{#1/#2}}

\newcommand\cont{\operatorname{cont}}
\newcommand\diff{\operatorname{diff}}

%%%%%%%%%%%%%%%%%%%%%%%%%%%%%%%%%%%%%%%%%%%%%%%%%%%%%%%%%%%%%%%%%%%%%%%%%

\section{Introduction}
In 1770 Lagrange proved that any quadratic irrational has a continued fraction expansion which is periodic after a certain stage and the converse was also proved, see \cite{Khinchin} or \cite{Olds}. Precisely:
\begin{thm}(\cite{Khinchin}, p. 48)
Every periodic continued fraction represents a quadratic irrational number and every quadratic irrational number is represented by a periodic continued fraction.
\end{thm}
The above mentioned continued fractions are all simple. Now introduce the continued fraction
\begin{equation*}
  F(x,y)=[x,y^{-1},x^{2},y^{-2},x^{3},y^{-3},x^{4},y^{-4},\dots]
\end{equation*}
where $x$ and $y$ are integers. Despite its simple form, this two variable continued fraction does not seem to appear in  the literature.  Some preliminary computer experiments suggest;

\begin{con}
If $x$ and $y$ are positive integers, and $x\neq y$, then $F(x,y)$ is transcendental.
\end{con}
However the case where $x=y$ is different. Section \ref{conv} makes an in-depth study of $F(x,x)$. In particular we prove that for $x$ a positive integer, $F(x,x)$ is the largest root of $P(x,z)=z^{2}-(2x-1)z-x$.

The original motivation for this work came from the study of simple continued fractions with many large partial quotients. A  classical example is the continued fraction expansion for $\exp(\frac{1}{x})$ given by Euler in the 1730's namely,
\begin{equation}\label{eq1}
  \sum_{k=0}^{\infty}\frac{x^{-k}}{k!}=[1, x-1, 1, 1, 3x-1, 1, 1, 5x-1, 1, ..., 1, (2n+1)x-1, 1, ...]
\end{equation}
for $x$ a large integer, there will be infinitely many large partial quotients, these are partial quotients that are linear in $x$.

The phenomenon of large partial quotients also occurs for polynomial functions of $x$. For example, one may simply truncate the series of (\ref{eq1}). For another example, consider the polynomial $$P(a,x)=3a(x+1) +x^2+3x+5$$ where $a$ is a very large positive integer. If we set $a=10^6$,  it has a root $x_{0}$ near $x=-1$ and a root $x_{1}$ near $x=-3000000$. The continued fraction of $-x_{0}$ begins

\begin{equation}\label{eq3}
  [1,10^6,3,\frac{10^6-1}{3},9,\frac{10^6-1}{9},27,\frac{10^6-1}{27},81,12345,1,2,26,1,2,4114,1,8,\dots]
\end{equation}
The other root of the quadratic has a similar pattern and is equal to $-3000000-r$ where the continued fraction of $r$ begins
\begin{equation}\label{eq3b}
  [1,1,10^6-1,3,\frac{10^6-1}{3},9,\frac{10^6-1}{9},27,\frac{10^6-1}{27},81,12345,1,2,26,1,2,4114,1,8,\dots]
\end{equation}
By Lagrange's theorem these will eventually be periodic, with perhaps a very long period. However the remarkable initial pattern suggests the study of $F(x,y)$, at least for $x=y$. We use the language of van der Poorten and Shallit \cite{Shallit1}, and call a partial quotient (other than `the first' which may be zero) `inadmissible' if it is zero, negative or a fraction. Hence all partial quotients of a regular continued fraction expansion of a real number are admissible. From the initial pattern in (\ref{eq3}) we are led to an elegant continued fraction,
\begin{equation}\label{eq4}
  [1,10^6,3,\frac{10^6-1}{3},3^2,\frac{10^6-1}{3^2},3^3,\frac{10^6-1}{3^3},3^4,\frac{10^6-1}{3^4},\dots,3^k,\frac{10^6-1}{3^k},\dots]
\end{equation}
which has partial quotients that are inadmissible from some point on. From computations in Mathematica it seems to converge to the same real number as (\ref{eq3}).  Consider more generally, for nonzero positive integers $x$ and $s$, the continued fraction
\begin{equation}\label{eq5a}
  F_{k}(x,s):=[x,\frac{s}{x},x^2,\frac{s}{x^2},x^3,\frac{s}{x^3},x^4,\frac{s}{x^4},\dots,x^k,\frac{s}{x^k}].
\end{equation}
We have;
\begin{thm}
The continued fraction $F_{k}(x,s)$ converges to the largest root of the polynomial $P(x,z,s)=sz^2-((s+1)x-1)z-x$.
\end{thm}
\noindent The proof of this theorem for $s=1$ is the subject of section \ref{conv}. Of special interest is the rate at which $P(x,F_{k}(x,s),s)$ approaches zero as $k\to\infty$; see equations (\ref{conveq1}) and (\ref{conveq2}) that occur in the proof of theorem \ref{thm15}.  Additional results on  (\ref{eq5a}) are presented in section \ref{gen}. We shall however begin with the case $s=1$,
\begin{equation}\label{eq5}
  F(x):=[x,\frac{1}{x},x^2,\frac{1}{x^2},x^3,\frac{1}{x^3},x^4,\frac{1}{x^4},\dots,x^k,\frac{1}{x^k},\dots]
\end{equation}
In section 2, we express the convergents of this continued fraction in terms of polynomials $A_{n}(x)$ and $B_{n}(x)$. Their properties are developed further in section 3.

Although $P(1,z,1)=0$ is the equation of the golden ratio, and (as we indicate later) one of the convergence proofs in section 3 is analogous to a result involving Fibonacci polynomials, we stress that these polynomials are of a different nature than Fibonacci and Chebyshev polynomials.

Section \ref{gen} gives a generalization $F(x,s)$ of of the continued fraction $F(x)$ and its properties. We conclude this section by drawing a connection between the continued fraction $F(1,y^{-1})$ and some $q-$series studied by Auluck \cite{Auluck}, and a Ramanujan $q-$series.

\section{Convergents and generating functions}\label{gen}
This section investigates the convergents of (\ref{eq5}), which gives two sequences $(A_j(x))$ and $(B_j(x))$ of polynomials with positive integer coefficients. The recurrence relations and generating functions of these sequences are studied. It is rather difficult to prove results on (\ref{eq5}) in its given form, but it becomes easier when we transform it into an `equivalent' continued fraction of the form;
\begin{equation}\label{eq6}
b_0+\frac{a_1}{b_1}\underset{+}{}\frac{a_2}{b_2}\underset{+}{}\frac{a_3}{b_3}\underset{+}{}\dots\underset{+}{}\frac{a_j}{b_j}\underset{+}{}\dots
\end{equation}
We use the familiar notation
\begin{equation*}\label{ }
   K\left(\frac{a_{j}}{b_{j}}\right)=K_{j=1}^{\infty}\left(\frac{a_{j}}{b_{j}}\right):=
   \frac{a_1}{b_1}\underset{+}{}\frac{a_2}{b_2}\underset{+}{}\frac{a_3}{b_3}\underset{+}{}\dots\underset{+}{}\frac{a_j}{b_j}\underset{+}{}\dots
\end{equation*}
Two continued fractions  are said to be \textbf{equivalent} i.e. $ K(\frac{a_{j}}{b_{j}})\sim  K(\frac{c_{j}}{d_{j}})$, if they have the same sequence of classical approximants (or convergents) (\cite{Lisa}, pp 77). It is straightforward to verify  that $\sim$ is indeed an equivalence relation.
\begin{thm}\label{thm1}{(\cite{Lisa}, pp 77)}
$K\left(\frac{a_{j}}{b_{j}}\right)\sim  K\left(\frac{c_{j}}{d_{j}}\right)$ if and only if there exists a sequence $\{r_{j}\}$ of complex numbers with $r_{0}=1$ and $r_{j}\neq0$ for all $j\in\N$, such that, $$c_{j}=r_{j-1}r_{j}a_{j},\;\;\;d_{j}=r_{j}b_{j}\;\;\;for\;\;all\;\;j\in\N$$
\end{thm}
\begin{thm}\label{thmeq} $F(x)$ is equivalent to $\tilde{F}(x)$ defined as
\begin{equation}\label{eq7}
\tilde{F}(x):=x+\frac{x}{1}\underset{+}{}\frac{1}{x}\underset{+}{}\frac{x}{1}\underset{+}{}\dots\underset{+}{}\frac{1}{x}\underset{+}{}
\frac{x}{1}\underset{+}{}\dots
\end{equation}
\end{thm}
\begin{proof}
In the notation of (\ref{eq6}), $F(x)$ is expressed as
$$F(x):=x+\frac{1}{1/x}\underset{+}{}\frac{1}{x^{2}}\underset{+}{}\frac{1}{1/{x^{2}}}\underset{+}{}\dots\underset{+}{}\frac{1}{x^{k}}\underset{+}{}
\frac{1}{1/{x^{k}}}\underset{+}{}\dots
$$ Let $K\left(\frac{a_{j}}{b_{j}}\right)=F(x)-x$ and apply theorem \ref{thm1} with the sequence $\{r_{j}\}$ defined as
$$r_{j}:=\left\{\begin{array}{ll}
                       x^\frac{{j+1}}{2} & for\;\;j\;\;odd \\
                       x^{-\frac{{j}}{2}} & for\;\;j\;\;even
                    \end{array}
  \right.$$ %$F(x)$ is equivalent to $\tilde{F}(x)$.
\end{proof}
From the relation $\frac{A_{j+1}}{B_{j+1}}=\frac{b_{j+1}A_{j}+a_{j+1}A_{j-1} }{b_{j+1}B_{j}+a_{j+1}B_{j-1}  }$ of the convergents $\frac{A_j}{B_j}$  of  (\ref{eq6}) we have the recurrence relations
\begin{equation}\label{eq8}
  \begin{array}{ccc}
    A_{j+1}(x) & =& b_{j+1}A_{j}(x)+a_{j+1}A_{j-1}(x) \\
    B_{j+1}(x) & =& b_{j+1}B_{j}(x)+a_{j+1}B_{j-1}(x)
  \end{array}
\end{equation}
This will be useful in proving a similar result for (\ref{eq5}). The first 11 terms for $A_{j}(x)$ and $B_{j}(x)$ are given in  table \ref{tb1}.

\begin{table}\label{tb1}
$$\begin{array}{|l||l||l|}
\hline j &A_{j}&B_{j} \\
\hline 0 &  x & 1 \\
1 & 2x & 1 \\
 2 &  x(2x+1) & x+1  \\
 3 &  x(4x+1) & 2x+1  \\
 4 &  x (4 x^2 + 3 x + 1) & 2 x^2  + 2 x + 1 \\
 5 &  x (8 x^2  + 4 x + 1) & 4 x^2  + 3 x + 1 \\
 6 &  x (8 x^3 + 8 x^2+ 4 x + 1 ) & 4 x^3 + 5 x^2 + 3 x + 1 \\
 7 &  x (16 x^3 + 12 x^2 + 5 x + 1) & 8 x^3 + 8 x^2 + 4 x + 1 \\
 8 &  x (16 x^4+ 20 x^3+ 13 x^2+ 5 x + 1) & 8 x^4+ 12 x^3+ 9 x^2+ 4 x + 1 \\
 9 &  x (32 x^4+ 32 x^3+ 18 x^2+ 6 x + 1) & 16 x^4+ 20 x^3+ 13 x^2+ 5 x + 1  \\
 10 & x (32 x^5+ 48 x^4+ 38 x^3+ 19 x^2+ 6 x + 1 ) & 16 x^5+ 28 x^4+ 25 x^3+ 14 x^2+ 5 x + 1  \\\hline
\end{array}$$
\caption{First 11 polynomials for $A_{j}$ and $B_{j}$. }
\end{table}

\begin{thm}\label{thm2}
The ${A_{j}(x)}$ and ${B_{j}(x)}$ from the convergents $\frac{A_{j}(x)}{B_{j}(x)}$ of (\ref{eq7}) are given by the recurrence relations;
\begin{equation}\label{eq9}
\begin{array}{ccc}
   A_{j}(x) &=& (2x+1)A_{j-2}(x)-xA_{j-4}(x)\\
   B_{j}(x) &=& (2x+1)B_{j-2}(x)-xB_{j-4}(x)
\end{array}
\end{equation}
\end{thm}
\begin{proof}
To simplify the notation we will use ${A_j}$ and ${B_j}$. From (\ref{eq7})  first observe that for $j$ odd, $a_{j}=x$ and $b_{j}=1$, and for $j$ even, $a_{j}=1$ and $b_{j}=x$.  Let $\frac{A_j}{B_j}$ be the convergents of $\tilde{F}(x)$. The $A_{j}$ and $B_j$ have initial conditions $A_{0}=x$, $A_{1}=2x$, $A_{2}=x(2x+1)$, $B_{0}=B_{1}=1$ and $B_{2}={x+1}$. Also for $j=4$, (\ref{eq9}) holds. Suppose that for all $i\leq j$ (\ref{eq9}) is true. Then for $j$ even,
\begin{eqnarray*}
  \frac{A_{j+2}}{B_{j+2}}:=\tilde{z}_{j+2}(x)
   &=& x+\frac{x}{1}\underset{+}{}\frac{1}{x}\underset{+}{}\frac{x}{1}\underset{+}{}\dots\underset{+}{}\frac{1}{x}\underset{+}{}\frac{x^2}{x+1}\\
   &=& \frac{(x+1)A_{j}+x^{2}A_{j-1}}{(x+1)B_{j}+x^{2}B_{j-1}}.
\end{eqnarray*}
Since $j$ is even, by (\ref{eq8})  $A_{j}=xA_{j-1}+A_{j-2}$ and $B_{j}=xB_{j-1}+B_{j-2}$ so that,
\begin{eqnarray*}
   A_{j+2}&=& (x+1)A_{j}+x(A_{j}-A_{j-2})  \\
    &=& (2x+1)A_{j}-xA_{j-2},
\end{eqnarray*}
%and
\begin{eqnarray*}
   B_{j+2}&=& (x+1)B_{j}+x(B_{j}-B_{j-2})  \\
    &=& (2x+1)B_{j}-xB_{j-2}.
\end{eqnarray*}
Similarly for $j$ odd,
\begin{eqnarray*}
  \frac{A_{j+2}}{B_{j+2}}:=\tilde{z}_{j+2}(x)
   &=& x+\frac{x}{1}\underset{+}{}\frac{1}{x}\underset{+}{}\frac{x}{1}\underset{+}{}\dots\underset{+}{}\frac{x}{1}\underset{+}{}\frac{1}{2x}\\
   &=& \frac{2xA_{j}+A_{j-1}}{2xB_{j}+B_{j-1}}
\end{eqnarray*}
and by (\ref{eq8})  $A_{j}=A_{j-1}+xA_{j-2}$ and $B_{j}=B_{j-1}+xB_{j-2}$. Thus
\begin{eqnarray*}
   A_{j+2} &=& 2xA_{j}+(A_{j}-xA_{j-2}) \\
    &=& (2x+1)A_{j}-xA_{j-2},
\end{eqnarray*}
 and
 \begin{eqnarray*}
   B_{j+2} &=& 2xB_{j}+(B_{j}-xB_{j-2}) \\
    &=& (2x+1)B_{j}-xB_{j-2}.
\end{eqnarray*}
We have the same recurrence for $j$ odd and $j$ even. Hence by induction the recurrence holds for all $j$. This is the recurrence for $F(x)$ by the equivalence $F(x)\sim \tilde{F}(x)$.
\end{proof}

It is clear from the initial conditions and from the recurrence relations (\ref{eq9}) that $A_j$'s and $B_j$'s are polynomials in $x$ with positive integer coefficients. We now give their generating functions.
\begin{prop}
$A_{2j}$ and $A_{2j+1}$ have generating functions
\begin{eqnarray}
  \label{eq10} \frac{x}{1-(2x+1)t+x t^2}  &=& \sum_{j=0}^{\infty}A_{2j}(x)t^{j}, \\
  \label{eq11} \frac{x (2-t)}{1-(2x+1)t+x t^2} &=& \sum_{j=0}^{\infty}A_{2j+1}(x)t^{j},
\end{eqnarray}
while $B_{2j}$ and $B_{2j+1}$ have generating functions
\begin{eqnarray*}
   \frac{1}{1-(2x+1)t+xt^2}  &=& \sum_{j=0}^{\infty}B_{2j+1}(x)t^{j}, \\
   \frac{(1-xt)}{1-(2x+1)t+xt^2} &=& \sum_{j=0}^{\infty}B_{2j}(x)t^{j}.
\end{eqnarray*}
\end{prop}
\begin{proof}
From equation (\ref{eq9})
\begin{eqnarray*}
  \sum_{j=2}^{\infty}A_{2j}t^{j} &=& \sum_{j=2}^{\infty}\left((2x+1)A_{2j-2}-xA_{2j-4}\right)t^{j} \\
    &=& (2n+1)x\sum_{j=2}^{\infty}A_{2j-2}t^{j-1}-xt^{2}\sum_{j=2}^{\infty}A_{2j-4}t^{j-2}  \\
    &=& (2n+1)x\sum_{j=1}^{\infty}A_{2j-2}t^{j-1}-xt^{2}\sum_{j=2}^{\infty}A_{2j-4}t^{j-2}-(2x+1)tA_{0}
\end{eqnarray*}
hence
\begin{eqnarray*}
  \left(1-(2x+1)t+xt^{2}\right)\sum_{j=0}^{\infty}A_{2j}t^{j} &=&-(2x+1)tA_{0}+A_{0}+tA_{2}. \\
      &=& x
\end{eqnarray*}
The proof of (\ref{eq11}) and of $B_j$'s follows the same pattern as above.
\end{proof}
From the generating functions, we can observe relationships between the $A_j$ and $B_j$ polynomials. Here are some notable ones.

\begin{cor}
\begin{eqnarray}
  \label{eq12} A_{2j+1} &=& 2A_{2j}-A_{2j-2}\\
 %\label{eq13} A_{2j} &=& A_{2j+1}-xA_{2j-1} \\
  \label{eq14} B_{2j+1} &=& \frac{1}{x} A_{2j} \\
   \label{eq15} B_{2j} &=& \frac{1}{x} A_{2j}-A_{2j-2}\\
   \label{eq16}        &=& B_{2j+1}-xB_{2j-1}
\end{eqnarray}
\end{cor}
\noindent There are other relationships of interest between $B_{j}$'s and $\frac{1}{x}A_{j}$'s, for example;
\begin{lem}\label{lemrel1}
Let $A_{j}$ and $B_{j}$ be defined as before (see theorem \ref{thm2}). Then
\begin{equation}\label{eq17}
  B_{j}=\frac{1}{2x}A_{j}+\frac{1}{2}B_{j-2}
\end{equation}
\begin{proof}
Proceed by induction on $j$ with base case $j=2$ using the initial conditions for $A_j$ and $B_{j}$ as well as their recurrence
relations. For $j=2$, $x+1=\frac{1}{2x}A_{2}+\frac{1}{2}B_{0}=B_{2}$. Suppose for all $i\leq j$, (\ref{eq17}) is true. Then
\begin{eqnarray*}
  B_{j+1} &=& (2x+1)B_{j-1}-xB_{j-3} \\
    &=& (2x+1)\left(\frac{1}{2x}A_{j-1}+\frac{1}{2}B_{j-3}\right)-x\left(\frac{1}{2x}A_{j-3}+\frac{1}{2}B_{j-5}\right) \\
    &=&\frac{1}{2x}\left((2x+1)A_{j-1}-x A_{j-3}\right)-\frac{1}{2}\left((2x+1)B_{j-3}-xB_{j-5}\right) \\
    &=& \frac{1}{2x}A_{j+1}+\frac{1}{2}B_{j-1}
\end{eqnarray*}
\end{proof}
\end{lem}
We use the above lemma to show that the $B_{j}$ polynomials can be expressed in terms of the $A_{j}$ polynomials.
\begin{thm}
Let $A_{j}$ and $B_{j}$ be defined as before, then
\begin{eqnarray}
 \label{eq18} B_{2k} &=& \frac{A_{0}}{2^{k}x}+\frac{1}{x}\sum_{j=1}^{k}\frac{A_{2j}}{2^{k-j+1}} \\
 \label{eq19} B_{2k+1} &=& \frac{1}{x}\sum_{j=0}^{k}\frac{A_{2j+1}}{2^{k-j+1}}
\end{eqnarray}
\end{thm}
\begin{proof}
From the initial conditions, $\frac{A_{0}}{x}=B_{0}$ and $\frac{A_{0}}{2x}+\frac{A_{2}}{2x}=B_{2}$. Suppose for $i\leq j$ (\ref{eq18}) is true. Then by by (\ref{eq17})
\begin{eqnarray*}
  B_{2k+2} &=& \frac{1}{2x}A_{2k+2}+\frac{1}{2}B_{2k} \\
    &=& \frac{1}{2x}A_{2k+2}+\frac{1}{2}\left( \frac{A_{0}}{2^{k}x}+\frac{1}{x}\sum_{j=1}^{k}\frac{A_{2j}}{2^{k-j+1}} \right) \\
    &=&  \frac{A_{0}}{2^{k+1}x}+ \frac{1}{x}\left( \frac{A_{2k+2}}{2}+\sum_{j=1}^{k}\frac{A_{2j}}{2^{k-j+2}} \right)\\
    &=& \frac{A_{0}}{2^{k+1}x}+\frac{1}{x}\sum_{j=1}^{k+1}\frac{A_{2j}}{2^{k-j+2}}.
\end{eqnarray*}
Similarly, $\frac{A_{1}}{2x}=B_{1}$ and suppose for $i\leq j$ (\ref{eq19}) is true. Then by by (\ref{eq17})
\begin{eqnarray*}
  B_{2k+3} &=& \frac{1}{2x}A_{2k+3}+\frac{1}{2}B_{2k+1} \\
    &=& \frac{1}{2x}A_{2k+3}+\frac{1}{x}\sum_{j=0}^{k}\frac{A_{2j+1}}{2^{k-j+2}} \\
   % &=&   \frac{1}{n}\left( \frac{A_{2k+1}}{2n}+\sum_{j=0}^{k}\frac{A_{2j+1}}{2^{k-j+2}} \right)\\
    &=& \frac{1}{x}\sum_{j=0}^{k+1}\frac{A_{2j+1}}{2^{k-j+2}},
\end{eqnarray*}
\noindent so by induction on $j$, (\ref{eq18}) and (\ref{eq19}) are true for all $j$.
\end{proof}

%..........................................................................................................................................................

\section{Convergence of $F_{k}(x)$}\label{conv}
 \noindent In this section we prove that $F_{k}(x)$ converges to a quadratic irrational. We present two different proofs for the convergence.  Since $\tilde{F}(x)$ is periodic, if it converges, the proof of its convergence is straightforward. By the equivalence proved in theorem \ref{thmeq}, $F(x)$ will converge to the same limit.%we expect its limit to be same as that of.
\begin{prop}
If $\tilde{F}(x)$ converges, then the limit is a root of the quadratic $P(x,z)=z^{2}-(2x-1)z-x$.
\end{prop}
\begin{proof}
Suppose $\tilde{F}(x)$  converge to $z$, then,
\begin{eqnarray*}
  z &=& x+\frac{x}{1}\underset{+}{}\frac{1}{x}\underset{+}{}\frac{x}{1}\underset{+}{}\dots\underset{+}{}\frac{1}{x}\underset{+}{}
\frac{x}{1}\underset{+}{}\dots \\
    &=& x+\frac{x}{1}\underset{+}{}\frac{1}{z}
\end{eqnarray*}
from which we get $z^2+(1-2x)z-x=0$.
\end{proof}

Let $\alpha$ be an irrational number with a simple continued fraction $$\alpha=[a_{0},a_{1},a_{2},\dots].$$ We know by Lagrange that the continued fraction is periodic and does not terminate. Let $p_{n}/q_{n}$ be the convergents of the continued fraction for $\alpha$. Call $p_{n}/q_{n}$ an even convergent if $n$ is even, and an odd convergent if $n$ is odd.  By theorem 7.8 of (\cite{Stark}, p 193), $$\frac{p_{0}}{q_{0}}<\frac{p_{2}}{q_{2}}<\frac{p_{4}}{q_{4}}<\dots<\alpha<\dots<\frac{p_{5}}{q_{5}}<\frac{p_{3}}{q_{3}}<\frac{p_{1}}{q_{1}}.$$ Even convergents are strictly increasing while odd convergent are strictly decreasing.

We now prove the convergence of $F_{k}(x)$ by solving the recurrences (\ref{eq9}) by the method of characteristic roots (see for example \cite{Mott} p 300). In particular, we treat even and odd convergents of $F_{k}(x)$ separately.

\begin{thm}\label{thmzk}
Let $F_{k}(x)$ be defined as below for $k\geq1$,
\begin{equation}\label{eq20}
  F_{k}(x):=[x,\frac{1}{x},x^2,\frac{1}{x^2},x^3,\frac{1}{x^3},x^4,\frac{1}{x^4},\dots,x^k,\frac{1}{x^k}]
\end{equation}
Then $F_{k}(x)$ converges to a quadratic irrational that is the positive root of the polynomial $P(x,z)=z^{2}-(2x-1)z-x$.
\end{thm}
\begin{proof}
Since odd convergents are increasing and even convergents are decreasing, they will be treated differently. Let $a_{k}=A_{2k}$. Then from (\ref{eq9}) we get the recurrence $a_{k}=(2x+1)a_{k-1}-xa_{k-2}$ with characteristic polynomial $\lambda^{2}-(2x+1)\lambda+x=0$. Roots of this polynomial are  $\lambda_{1}=\frac{1}{2}(2x+1+\sqrt{4x^{2}+1})$ and $\lambda_{2}=\frac{1}{2}(2x+1-\sqrt{4x^{2}+1})$, both positive with $\lambda_{1}>\lambda_{2}$ and $a_{k}=\alpha \lambda_{1}^{k}+\beta \lambda_{2}^{k}$.

From the initial conditions $A_{0}=x$ and $A_2=2x^2+x$, we get the values of the constants $\alpha=\frac{\lambda_{2}x-2x^{2}-x}{-(\lambda_{1}-\lambda_{2})}$ and $\beta=\frac{\lambda_{1}x-2x^{2}-x}{\lambda_{1}-\lambda_{2}}$ so that
$$a_{k}=\frac{\lambda_{2}x-2x^{2}-x}{-(\lambda_{1}-\lambda_{2})} \lambda_{1}^{k}+\frac{\lambda_{1}x-2x^{2}-x}{\lambda_{1}-\lambda_{2}}\lambda_{2}^{k}.$$

Similarly let $b_{k}=B_{2k}$. Then from (\ref{eq9}), and the fact that the recurrences are the same,  $b_{k}=\alpha' \lambda_{1}^{k}+\beta' \lambda_{2}^{k}$. From the initial conditions $B_{0}=1$ and $B_2=x+1$, we get  $\alpha'=\frac{\lambda_{2}-x-1}{-(\lambda_{1}-\lambda_{2})}$ and $\beta'=\frac{\lambda_{1}-x-1}{\lambda_{1}-\lambda_{2}}$ so that
$$b_{k}=\frac{\lambda_{2}-x-1}{-(\lambda_{1}-\lambda_{2})} \lambda_{1}^{k}+\frac{\lambda_{1}-x-1}{\lambda_{1}-\lambda_{2}}\lambda_{2}^{k},$$
\begin{eqnarray*}
  \frac{a_{k}}{b_{k}} &=& \frac{(\lambda_{2}x-2x^{2}-x) \lambda_{1}^{k}-(\lambda_{1}x-2x^{2}-x) \lambda_{2}^{k}}{(\lambda_{2}-x-1)\lambda_{1}^{k}-(\lambda_{1}-x-1)\lambda_{2}^{k}} \\
   &=&  \frac{(\lambda_{2}x-2x^{2}-x)-(\lambda_{1}x-2x^{2}-x) (\lambda_{2}/ \lambda_{1})^{k}}{(\lambda_{2}-x-1)-(\lambda_{1}-x-1)(\lambda_{2}/ \lambda_{1})^{k}},
\end{eqnarray*}
 and recalling that $\lambda_{2}<\lambda_{1}$,
 \begin{equation}\label{eq21}
   \lim_{k\to\infty}\frac{a_{k}}{b_{k}}=\frac{\lambda_{2}x-2x^{2}-x}{\lambda_{2}-x-1}.
 \end{equation}
Substitute for $\lambda_{2}$ in equation \ref{eq21} and rationalize the denominator to get $$\lim_{k\to\infty}\frac{a_{k}}{b_{k}}=\frac{2x-1+\sqrt{4x^{2}+1}}{2}.$$

On the other hand, let $a_{k}=A_{2k-1}$. Since we have the same recurrence relation as before, $a_{k}=\gamma \lambda_{1}^{k}+\delta \lambda_{2}^{k}$. From the initial conditions $A_{1}=2x$ and $A_3=4x^2+x$, we get  $\gamma=\frac{2x\lambda_{2}-4x^{2}-x}{(\lambda_{1}\lambda_{2}-\lambda_{1}^{2})}$ and $\delta=\frac{2x\lambda_{1}-4x^{2}-x}{\lambda_{1}\lambda_{2}-\lambda_{2}^{2}}.$ Thus
$$a_{k}=\frac{2x\lambda_{2}-4x^{2}-x}{\lambda_{1}\lambda_{2}-\lambda_{1}^{2}} \lambda_{1}^{k}+\frac{2x\lambda_{1}-4x^{2}-x}{\lambda_{1}\lambda_{2}-\lambda_{2}^{2}}\lambda_{2}^{k}.$$

Similarly let $b_{k}=B_{2k-1}$. From (\ref{eq9}) and the fact that the recurrences are the same,  $b_{k}=\gamma' \lambda_{1}^{k}+\delta' \lambda_{2}^{k}.$ From the initial conditions $B_{1}=1$ and $B_{3}=2x+1$, we get  $\gamma'=\frac{\lambda_{2}-2x-1}{\lambda_{1}\lambda_{2}-\lambda_{1}^{2}}$ and $\delta'=\frac{\lambda_{1}-2x-1}{\lambda_{1}\lambda_{2}-\lambda_{2}^{2}}.$ Thus
$$b_{k}=\frac{\lambda_{2}-2x-1}{\lambda_{1}\lambda_{2}-\lambda_{1}^{2}} \lambda_{1}^{k}+\frac{\lambda_{1}-2x-1}{\lambda_{1}\lambda_{2}-\lambda_{2}^{2}}\lambda_{2}^{k}.$$
Hence
\begin{equation}\label{eq22}
   \lim_{k\to\infty}\frac{a_{k}}{b_{k}}=\frac{2x\lambda_{2}-4x^{2}-x}{\lambda_{2}-2x-1}.
 \end{equation}
Substitute $\lambda_{2}$ into (\ref{eq22}) and rationalize the denominator to get $$\lim_{k\to\infty}\frac{a_{k}}{b_{k}}=\frac{2x-1+\sqrt{4x^{2}+1}}{2}.$$ This completes the proof.
\end{proof}

In preparation for an inductive proof of theorem \ref{thmzk}, we give some results on linear combinations of products of the $A_{j}$ and $B_{j}$ polynomials that yield monomials.
\begin{lem} Let $\frac{A_{j}}{B_{j}}$ be the convergents of $F(n)$. For $j\geq 1,$
\begin{eqnarray}
% \nonumber to remove numbering (before each equation)nn
 \label{eq23} A_{2j-2}A_{2j+2}-A_{2j}^{2} &=& -x^{j+2}, \\
 \nonumber B_{2j-2}B_{2j+2}-B_{2j}^{2} &=& x^{j+1},
\end{eqnarray}
and for $j\geq 2$,
\begin{eqnarray}
  \label{eq24} A_{2j-3}A_{2j+1}-A_{2j-1}^{2} &=& x^{j}, \\
  \nonumber B_{2j-3}B_{2j+1}-B_{2j-1}^{2} &=& -x^{j-1}.
\end{eqnarray}
\end{lem}
\begin{proof}
From the initial conditions $A_{0}=x,\;\;A_{2}=2x^{2}+x$ and $A_{4}=4x^{3}+3x^{2}+x$, $A_{0}A_{4}-A_{2}^{2}=-x^{3}.$
Suppose (\ref{eq23}) holds for all integers up to $j$. Then
\begin{eqnarray*}
  A_{2j}A_{2j+4}-A_{2j+2}^{2} &=& A_{2j}\{(2x+1)A_{2j+2}-xA_{2j}\}-A_{2j+2}^{2} \\
    &=& A_{2j+2}\{(2x+1)A_{2j}\}-xA_{2j}^{2}-A_{2j+2}^{2} \\
    &=& A_{2j+2}\{A_{2j+2}+xA_{2j-2}\}-xA_{2j}^{2}-A_{2j+2}^{2} \\
    &=& xA_{2j-2}A_{2j+2}-xA_{2j}^{2} \\
    &=& -x^{j+3}
\end{eqnarray*}
where we have used the recurrence relations (\ref{eq9}) for $A_{j}$. By induction on $j$, equation (\ref{eq23}) is true for all $j$. The proofs of the other relations are identical to the above proof with the only difference being the initial conditions.
\end{proof}
\begin{rem}
This provides an alternative definition of $A_{j}(x)$ and $B_{j}(x)$ by non-linear recurrences. For $x=1$ this reduces to the well known non-linear recurrences that define the Fibonacci numbers.
\end{rem}
\begin{lem}\label{mainlem} Let $A_{j}$ and $B_{j}$ be as before. Then for $j\geq 0,$
\begin{equation}\label{eq25}
  A_{2j}^{2}-(2x-1)A_{2j}B_{2j}-xB_{2j}^{2} = -x^{j+2},
\end{equation}
and for $j\geq1,$
\begin{equation}\label{eq26}
  A_{2j-1}^{2}-(2x-1)A_{2j-1}B_{2j-1}-xB_{2j-1}^{2} = x^{j}.
\end{equation}
\end{lem}
\begin{proof}
%The proof is by induction on $j$.
From the initial conditions $A_{0}=x$ and $B_{0}=1$,   $A_{0}^{2}-(2x-1)A_{0}B_{0}-xB_{0}^{2}=-x^{2}$. Suppose equation (\ref{eq25}) is true for all $i\leq j$. By using the recurrence relations (\ref{eq9}),
$$A_{2j+2}^{2}-(2x-1)A_{2j+2}B_{2j+2}-xB_{2j+2}^{2}=[(2x+1)A_{2j}-xA_{2j-2}]^{2}$$
\begin{eqnarray}
   \nonumber && -(2x-1)((2x+1)A_{2j}-xA_{2j-2})((2x+1)B_{2j}-xB_{2j-2})-x((2x+1)B_{2j}-xB_{2j-2})^{2} \\
   \nonumber &=& (2x+1)^{2}(A_{2j}^{2}-(2x-1)A_{2j}B_{2j}-xB_{2j}^{2})+x^{2}(A_{2j-2}^{2}-(2x-1)A_{2j-2}B_{2j-2}-xB_{2j-2}^{2})  \\
   \nonumber && -(2x+1)(2xA_{2j-2}A_{2j}-(2x-1)(xA_{2j}B_{2j-2}+xA_{2j-2}B_{2j})-2x^{2}B_{2j-2}B_{2j}) \\
   \nonumber &=& -(2x+1)^{2}x^{j+2}-x^{j+3} \\%-(2x+1)[2xA_{2j-2}A_{2j}-(2x-1)(xA_{2j}B_{2j-2}+xA_{2j-2}B_{2j})-2x^{2}B_{2j-2}B_{2j}]  \\
   \nonumber && -(2x+1)[2xA_{2j-2}A_{2j}-(2x-1)(xA_{2j}B_{2j-2}+xA_{2j-2}B_{2j})-2x^{2}B_{2j-2}B_{2j}].
\end{eqnarray}
It suffices to show that the square bracket  equals $-(2n+1)n^{j+2}$. Using (\ref{eq15}) and (\ref{eq23}) we get
$$2xA_{2j-2}A_{2j}-(2x-1)(xA_{2j}B_{2j-2}+xA_{2j-2}B_{2j})-2x^{2}B_{2j-2}B_{2j}$$
\begin{eqnarray}
  \nonumber  &=& 2xA_{2j-2}A_{2j}-(2x-1)(A_{2j}(A_{2j-2}-xA_{2j-4})+A_{2j-2}(A_{2j}-xA_{2j-2}))\\
  \nonumber  &&  -2(A_{2j-2}-xA_{2j-4})(A_{2j}-xA_{2j-2}) \\
  \nonumber  &=& -2xA_{2j-2}A_{2j}+x(2x+1)(A_{2j-2}^{2}+A_{2j-4}A_{2j})-2x^{2}A_{2j-4}A_{2j-2} \\
  \nonumber  &=& -2xA_{2j-2}(A_{2j}+xA_{2j-4})+x(2x+1)(2A_{2j-2}-x^{j+1}) \\
  \nonumber &=& -2xA_{2j-2}((2x+1)A_{2j-2})+x(2x+1)(2A_{2j-2}-x^{j+1}) \\
  \nonumber  &=& -(2x+1)x^{j+2}.
\end{eqnarray}
Substituting back gives $$A_{2j+2}^{2}-(2x-1)A_{2j+2}B_{2j+2}-xB_{2j+2}^{2}=-x^{j+3}.$$% hence by induction on $j$ equation (\ref{eq21}) is true for all $j$.

Similarly, $A_{1}=2x$ and $B_{1}=1$ so that $A_{1}^{2}-(2x-1)A_{1}B_{1}-xB_{1}^{2}=x$. Suppose equation (\ref{eq26}) is true for all $i\leq j$. By using the recurrence relations (\ref{eq9}), we get
$$A_{2j+1}^{2}-(2x-1)A_{2j+1}B_{2j+1}-xB_{2j+1}^{2}=[(2x+1)A_{2j-1}-xA_{2j-3}]^{2}$$
\begin{eqnarray}
   \nonumber && -(2x-1)((2x+1)A_{2j-1}-xA_{2j-3})((2x+1)B_{2j-1}-xB_{2j-3})-x((2x+1)B_{2j-1}-xB_{2j-3})^{2} \\
   \nonumber &=& (2x+1)^{2}(A_{2j-1}^{2}-(2x-1)A_{2j-1}B_{2j-1}-xB_{2j-1}^{2})+x^{2}(A_{2j-3}^{2}-(2x-1)A_{2j-3}B_{2j-3}-xB_{2j-3}^{2})  \\
   \nonumber && -(2x+1)(2xA_{2j-3}A_{2j-1}-(2x-1)(xA_{2j-1}B_{2j-3}+xA_{2j-3}B_{2j-1})-2x^{2}B_{2j-3}B_{2j-1}) \\
   \nonumber &=& (2x+1)^{2}x^{j}+x^{j+1} \\%-(2x+1)[2xA_{2j-2}A_{2j}-(2x-1)(xA_{2j}B_{2j-2}+xA_{2j-2}B_{2j})-2x^{2}B_{2j-2}B_{2j}]  \\
   \nonumber && -(2x+1)[2xA_{2j-3}A_{2j-1}-(2x-1)(xA_{2j-1}B_{2j-3}+xA_{2j-3}B_{2j-1})-2x^{2}B_{2j-3}B_{2j-1}].
\end{eqnarray}
It suffices to show that the square bracket in the equation above equals $(2x+1)x^{j}$. Using (\ref{eq13}), (\ref{eq14}) and (\ref{eq24}), we get
$$2xA_{2j-3}A_{2j-1}-(2x-1)(xA_{2j-1}B_{2j-3}+xA_{2j-3}B_{2j-1})-2x^{2}B_{2j-3}B_{2j-1}$$
\begin{eqnarray}
  \nonumber  &=& 2xA_{2j-3}A_{2j-1}-(2x-1)(A_{2j-1}(A_{2j-3}-xA_{2j-5})+A_{2j-3}(A_{2j-1}-xA_{2j-3})) \\
  \nonumber  && -2(A_{2j-3}-xA_{2j-5})(A_{2j-1}-xA_{2j-3}) \\
  \nonumber  &=& -2xA_{2j-3}A_{2j-1}+x(2x+1)(A_{2j-3}^{2}+A_{2j-5}A_{2j-1})-2x^{2}A_{2j-5}A_{2j-3} \\
  \nonumber  &=& -2xA_{2j-2}(A_{2j}+xA_{2j-4})+x(2x+1)(2A_{2j-2}-x^{j+1}) \\
   \nonumber &=& -2xA_{2j-3}((2x+1)A_{2j-3})+x(2x+1)(2A_{2j-3}+x^{j-1}) \\
  \nonumber  &=& (2x+1)x^{j}.
\end{eqnarray}
Substituting back gives $$A_{2j+1}^{2}-(2x-1)A_{2j+1}B_{2j+1}-xB_{2j+1}^{2}=x^{j+1}.$$ By induction on $j$, (\ref{eq25}) and (\ref{eq26}) are true for all $j$.
\end{proof}
We now use lemma \ref{mainlem} to prove the convergence of $F_{k}(x)$ in a manner that quantifies how close the convergents are to the limit.

\begin{thm}\label{thm15}
For $k\geq1$, let
\begin{equation}\label{eq20}
  F_{k}(x):=[x,\frac{1}{x},x^2,\frac{1}{x^2},x^3,\frac{1}{x^3},x^4,\frac{1}{x^4},\dots,x^k,\frac{1}{x^k}],
\end{equation}
 and
\begin{equation}\label{eq28}
  F_{k}^{*}(x):=[x,\frac{1}{x},x^2,\frac{1}{x^2},x^3,\frac{1}{x^3},x^4,\frac{1}{x^4},\dots,x^k]
\end{equation} be truncations of $F(x)$ to give odd and even convergents respectively. Then $F_{k}(x)$ and $F_{k}^{*}(x)$ converge to the positive root of $P(x,z)=z^{2}-(2x-1)z-x$.
\end{thm}
\begin{proof}
Since $F_{k}(x)=\frac{A_{2k-1}}{B_{2k-1}}$,  equation (\ref{eq26}) yields
\begin{eqnarray}\label{conveq1}
  \nonumber P(x,F_{k}(x)) &=& \frac{A_{2k-1}^{2}-(2x-1)A_{2k-1}B_{2k-1}-xB_{2k-1}^{2}}{B_{2k-1}^{2}} \\
   &=& \frac{x^{k}}{B_{2k-1}^{2}(x)}.
\end{eqnarray}
Here $B_{2k-1}(x)$ has positive integer coefficients and $\deg(B_{2k-1}(x))=k-1$, so for $x>1$,
\begin{equation*}
  \lim_{k\to\infty} P(x,F_{k}(x)) =  \lim_{k\to\infty}\frac{x^{k}}{B_{2k-1}^{2}(x)}  \leq  \lim_{k\to\infty}\frac{1}{x^{k^{2}-3k+1}}  =0.
\end{equation*}
Since $F_{k}^{*}(x)=\frac{A_{2k-2}}{B_{2k-2}}$, equation (\ref{eq25}) yields
\begin{eqnarray}\label{conveq2}
  \nonumber P(x,F_{k}^{*}(x)) &=& \frac{A_{2k-2}^{2}-(2x-1)A_{2k-2}B_{2k-2}-xB_{2k-2}^{2}}{B_{2k-2}^{2}} \\
   &=& \frac{-x^{k+1}}{B_{2k-2}^{2}(x)}.
\end{eqnarray}
As before, $B_{2k}(x)$ has positive integer coefficients and $\deg(B_{2k-2}(x))=k-1$, so for $x>1$,
\begin{equation*}
  \lim_{k\to\infty} P(x,F_{k}^{*}(x))=  \lim_{k\to\infty}\frac{-x^{k+1}}{B_{2k-2}^{2}(x)} \leq  \lim_{k\to\infty}\frac{-1}{x^{k^{2}-3k}} =0.
\end{equation*}
To show that they converge to the positive root of $P(x,z)$, write $P(x,z)=(z-z_{1})(z-z_{2})$ where $z_{1}=\frac{1}{2}(2x-1+\sqrt{4x^{2}+1})$ and $z_{2}=\frac{1}{2}(2x-1-\sqrt{4x^{2}+1})$. Note that $F_{k}(x)$ and $F_{k}^{*}(x)$ are all positive and so they converge to $z_{1}$.
\end{proof}

The evaluations of the function $P(x,z)$ at the convergents of the continued fraction is analogous to the known result that $u^{2}-xu-1$ evaluated at $[x,\dots,x]$ with $n$ partial quotients of $x$ is $$\frac{(-1)^{n}}{F_{n}^{2}(x)}$$ where $F_{n}(x)$ is the Fibonacci polynomial of degree $n$. This result implies that $[x,x,x,\dots]$ converges to $\frac{1}{2}(x+\sqrt{4+x^{2}})$, a root of $u^{2}-xu-1$, for $x\geq1$.

%..............................................................................................................................................................

\section{A generalization of $F_{k}(x)$}\label{gen}
We now give a generalization of (\ref{eq5}) by introducing a parameter $`s$' to get a continued fraction of the form (\ref{eq4}). In this section, we just present results and omit the proofs since they have the same construction as the proofs already presented. For a positive integer $s$, define $F(x,s)$ as;
\begin{equation}\label{eq28}
  F(x,s):=[x,\frac{s}{x},x^2,\frac{s}{x^2},x^3,\frac{s}{x^3},x^4,\frac{s}{x^4},\dots,x^k,\frac{s}{x^k},\dots]
\end{equation}
then $F(x,s)$ has an equivalent presentation
\begin{equation}\label{eq29}
\tilde{F}(x,s):=x+\frac{x}{s}\underset{+}{}\frac{1}{x}\underset{+}{}\frac{x}{s}\underset{+}{}\dots\underset{+}{}\frac{1}{x}\underset{+}{}
\frac{x}{s}\underset{+}{}\dots
\end{equation}
It is straightforward to verify the equivalence using  theorem (\ref{thm1}) with the sequence $\{r_{j}\}$ defined earlier. Clearly, $\tilde{F}(x,s)$ is periodic and converges to a root of the polynomial $P(x,z,s)=sz^2-((s+1)x-1)z-x$. By the  equivalence to $\tilde{F}(x,s)$, $F(x,s)$ converges to the same limit. To see this, consider the convergents $\frac{A_{j}(x,s)}{B_{j}(x,s)}$ of $F(x,s)$. The polynomials $A_{j}(x,s)$ and $B_{j}(x,s)$  have initial conditions $A_{0}=x,\;\;A_{1}=x(s+1),\;\;A_{2}=x((s+1)x+1),\;\;B_{0}=1,\;\;B_{1}=s$ and $B_{2}=xs+1$ with recurrence relations
\begin{equation}\label{eq30}
\begin{array}{ccc}
   A_{j}(x,s) &=& ((s+1)x+1)A_{j-2}(x,s)-xA_{j-4}(x,s)\\
   B_{j}(x,s) &=& ((s+1)x+1)B_{j-2}(x,s)-xB_{j-4}(x,s)
\end{array}
% \frac{}{}=  \frac{}{}
\end{equation}
From the recurrence relations (\ref{eq30}), we have the generating functions,
\begin{eqnarray*}
   \frac{x}{1-((s+1)x+1)t+x t^2}  &=& \sum_{j=0}^{\infty}A_{2j}(x,s)t^{j} \\
   \frac{x ((s+1)-t)}{1-((s+1)x+1)t+x t^2} &=& \sum_{j=0}^{\infty}A_{2j+1}(x,s)t^{j} \\
   \frac{s}{1-((s+1)x+1)t+x t^2}  &=& \sum_{j=0}^{\infty}B_{2j+1}(x,s)t^{j} \\
   \frac{(1-xt)}{1-((s+1)x+1)t+x t^2} &=& \sum_{j=0}^{\infty}B_{2j}(x,s)t^{j}
\end{eqnarray*}
From the generating functions, we establish relationships between the $A_{j}(x,s)$ and $B_{j}(x,s)$ polynomials given by
\begin{eqnarray*}
  A_{2j+1}(x,s) &=& (s+1)A_{2j}(x,s)-A_{2j-2}(x,s) \\
  xB_{2j+1}(x,s) &=& sA_{2j}(x,s) \\
 sB_{2j}(x,s) &=& B_{2j+1}(x,s)-xB_{2j-1}(x,s) \\
     &=&\frac{1}{x} A_{2j}(x,s)-A_{2j-2}(x,s)
\end{eqnarray*}

It can be shown that  for $j\geq 0,$
\begin{equation}\label{eq31}
  A_{2j}(x,s)^{2}-(2x-1)A_{2j}(x,s)B_{2j}(x,s)-xB_{2j}(x,s)^{2} = -x^{j+2},
\end{equation}
and for $j\geq1,$
\begin{equation}\label{eq32}
  A_{2j-1}(x,s)^{2}-(2x-1)A_{2j-1}(x,s)B_{2j-1}(x,s)-xB_{2j-1}(x,s)^{2} = x^{j}s.
\end{equation}
Thus one finds from equation (\ref{eq32}) that for odd convergents $F_{k}(x,s)$
\begin{equation}
   P(x,F_{k}(x,s)) =  \frac{x^{k}s}{B_{2k-1}^{2}(x,s)}.
\end{equation}
Here, $B_{2k-1}(x,s)$ has positive integer coefficients, has $s$ as a factor and $\deg(B_{2k-1}(x))=k-1$. Hence for $x>1$,
\begin{equation*}
  \lim_{k\to\infty} P(x,F_{k}(x,s)) =  \lim_{k\to\infty}\frac{x^{k}s}{B_{2k-1}^{2}(x)}  \leq  \lim_{k\to\infty}\frac{1}{x^{k^{2}-3k+1}s}  =0.
\end{equation*}
Similarly one finds from equation (\ref{eq31}) that for even convergents $F_{k}^{*}(x,s)$
\begin{equation}
   P(x,F_{k}^{*}(x,s)) =  \frac{-x^{k+1}}{B_{2k-2}^{2}(x,s)}.
\end{equation}
As before, $B_{2k}(x,s)$ has positive integer coefficients and $\deg(B_{2k-2}(x))=k-1$. Hence for $x>1$,
\begin{equation*}
  \lim_{k\to\infty} P(x,F_{k}^{*}(x,s))=  \lim_{k\to\infty}\frac{-x^{k+1}}{B_{2k-2}^{2}(x,s)} \leq  \lim_{k\to\infty}\frac{-1}{x^{k^{2}-3k}} =0.
\end{equation*}
Since for all $k$, $F_{k}(x,s)$ and $F_{k}^{*}(x,s)$ are positive, $F(x,s)$ is positive and converges to the positive root of the quadratic $$P(x,z,s)=sz^2-((s+1)x-1)z-x.$$
This proves that
\begin{thm}
The continued fraction $$F(x,s):=[x,\frac{s}{x},x^2,\frac{s}{x^2},x^3,\frac{s}{x^3},x^4,\frac{s}{x^4},\dots,x^k,\frac{s}{x^k},\dots]$$ converges to the positive root of the quadratic $P(x,z,s)=sz^2-((s+1)x-1)z-x.$
\end{thm}

We conclude by drawing a relationship between  the continued fraction $F(1,x^{-1})$ and some $q-$series studied by Auluck, see \cite{Auluck} and a Ramanujan $q-$series. Define a continued fraction with $m$ interlacing geometric series by
\begin{equation}\label{eq34}
  F(m;x_{1},\dots,x_{m}):=[x_{1},\dots,x_{m},x_{1}^{2},\dots,x_{m}^{2},x_{1}^{3},\dots,x_{m}^{3},\dots]
\end{equation}
Aside from the present work, nothing seems to be known for $m\geq2$. We also don't see anything of immediate interest for $m\geq3$, or even the case $m=1$ (despite a superficial resemblance to the famous Rogers-Ramanujan continued fraction). However the case $m=2$, $x_{1}=1$ and $x_{2}=x$, namely;
\begin{equation}\label{eq35}
  F(2;1,x):=[1,x,1,x^{2},1,x^{3},1,x^{4},\dots]
\end{equation}
does seem of immediate interest. Write the $k^{th}$ truncation of $F(2;1,x)$ in reduced form as $\frac{N_{k}(x)}{D_{k}(x)}$. Here $(D_{k}(x))$ seems to converge, as a formal power series, to the $q-$series
\begin{equation*}\label{}
  \sum_{k=1}^{\infty}\frac{x^{\frac{k^{2}+k}{2}}}{((1-x)(1-x^2)\dots(1-x^{k-1}))^{2}(1-x^k)}
\end{equation*}
studied by Auluck in \cite{Auluck}. Coefficients of this power series are \textbf{$A001524$}\footnote[1]{As given in the Online Encyclopedia of Integer Sequences https://oeis.org/A001524}. On the other hand, reversing the coefficients of $D_{2k+1}(x)$, the sequence $(D_{k}(x))$ seems to converge as a formal power series to another $q-$series %\textcolor{blue}{{\emph{(I did not find the specific q-series but there is reference of it)}}}
\textcolor{black}{ studied by Auluck whose coefficients are \textbf{$A005576$}\footnote[2]{As given in the Online Encyclopedia of Integer Sequences https://oeis.org/A005576}}.

The sequence of polynomials $(N_{2k+1}(x))$ seems to converge, as a formal power series, to a Ramanujan $q-$series
\begin{equation*}\label{}
  \sum_{k=1}^{\infty}\frac{x^{\frac{k^{2}+k}{2}}}{((1-x)(1-x^2)\dots(1-x^k))^{2}}
\end{equation*}
which has coefficients \textbf{$A143184$}\footnote[3]{As given in the Online Encyclopedia of Integer Sequences https://oeis.org/A143184}.

Note: The $A_{n}(x,s)$ and $B_{n}(x,s)$ polynomials that appear in the convergents of $F_{k}(x,s)$ have an interesting root distribution. This together with related results will be a subject of another paper.

\section{Acknowledgements}
I thank Kenneth Stolarsky for his continued support, insightful comments and guidance in the writing of this  paper. I  also  thank Bruce Berndt for always pointing me in the right direction and giving me useful ideas for solving some of the problems.

\end{document}